\documentclass[12pt]{amsart}
\usepackage{color}
\usepackage{graphicx}
\usepackage{amssymb}
\usepackage{amsthm}
\usepackage{amsmath}
\usepackage[margin=1in]{geometry}
\usepackage{caption}
\usepackage{subcaption}
\usepackage{float}

\newtheorem{theorem}{Theorem}
\newtheorem{lemma}{Lemma}
\newtheorem{corollary}{Corollary}
\newtheorem{proposition}{Proposition}

\newcommand{\bea}{\begin{eqnarray*}}
\newcommand{\eea}{\end{eqnarray*}}
\newcommand{\ben}{\begin{eqnarray}}
\newcommand{\een}{\end{eqnarray}}
\newcommand{\beq}{\begin{equation}}
\newcommand{\eeq}{\end{equation}}

\newcommand{\supp}{\operatorname{supp}}

\newcommand{\Id}{\operatorname{Id}}

\renewcommand{\hat}[1]{\widehat{#1}}

\newcommand{\om}{\omega}

\allowdisplaybreaks[2]
\begin{document}

\title{BLOW UP of A HYPERBOLIC SQG MODEL}
\author{Hang Yang}
\email{hang.yang@rice.edu}
\address{Rice University
Department of Mathematics -- MS 136
P.O. Box 1892
Houston, TX 77005-1892}

\begin{abstract}
This paper studies of a variation of the hyperbolic blow up scenario suggested by Hou and Luo's recent numerical simulation \cite{HouLuo1}. In particular, we propose a "hyperbolic" surface quasi-geostrophic equation characterized by a incompressible velocity field with a modified Biot-Savart law. For this model, we will show finite time blow up for a wide class of initial data. 
\end{abstract} 
\maketitle

\section{introduction}
The study of fluid mechanics PDE traces back to Leonhard Euler when he deduced the famous Euler equation for motion of ideal fluid
\begin{align*}
\partial_t u+u\cdot \nabla u&=-\nabla p,\quad u(x,0)=u_0(x)\\ \nabla\cdot u&=0
\end{align*}
which is set in a domain $D\subset \mathbb{R}^d, d=2,3$ with no penetration through boundary $(u\cdot n)|_{\partial D}=0$. The vector field $u$ describes particle velocity at a given point and the scalar function $p$ represents the pressure. It is one of the most important PDE ever written and enters as a cornerstone into a great variety of science and engineering subjects. The equation in vorticity form can be obtained by taking curl of the original equation. In other words, if we set $\omega=\nabla \times u$ in 2D, the equation reads 
$$\partial_t \omega+ u\cdot\nabla\omega=0, \quad \omega(x,0)=\omega_0(x)$$
The velocity field $u$ relates to the vorticity $\omega$ via the Biot-Savart law
\begin{equation} \label{EulerBS}
u=\nabla^{\perp} (-\Delta_D)^{-1} \omega
\end{equation}
where $-\Delta_D$ is the Dirichlet Laplacian. The equation in 3D has more complicated Biot-Savart law (see e.g. \cite{Euciso}) and differs from the 2D equation by an extra term $\omega\cdot \nabla u$ on the right hand side. It turns out that without this term on the right-hand-side to cause vorticity stretching, 2D Euler has a simpler nature. Global regularity in a natural class $C^k(\overline{D})$ has been proved by Wolibner \cite{Wolibner} and H\"{o}lder and existence of global unique solution for rough initial data $\omega_0\in L^{\infty}$ has been shown by Yudovich \cite{Yudovich}. However, the global regularity of the 3D Euler equation, alongside with that of the Navier-Stokes equation (NS) are now the two major problems in fluid mechanics. The surface quasi-geostrophic equation (SQG) was then introduced in the study of geophysics by Constantin, Majda and Tabak \cite{ConstantinTabak}. The viscous SQG, mostly considered in $\mathbb{R}^2$ or $\mathbb{T}^2$, is given by
\begin{equation}  \label{viscousSQG}
\partial_t \omega+u\cdot\nabla \omega=(-\Delta) ^\gamma \omega\quad(0<\gamma\leq1), \quad \omega(x,0)=\omega_0(x)
\end{equation}
with the Biot-Savart law 
\begin{equation} \label{SQGBS}
u=\nabla^{\perp} (-\Delta)^{-1/2} \omega 
\end{equation}
The equation has many features in common with 3D Navier-Stokes equation (see also \cite{ConstantinTabak}). Regularity analyses of solutions and their dependence on dissipation parameter $\gamma$ have been established for equation (\ref{viscousSQG})(\ref{SQGBS}). In 1999, Constantin-Wu \cite{ConstantinWu} settled the subcritical case ($1/2<\gamma\leq 1$) by proving global regularity. Later in 2006, the critical case ($\gamma=1/2$) was resolved by two independent works: Kiselev-Nazarov-Volberg \cite{KNV} and Caffarelli-Vasseur \cite{CaffarelliVasseur}. Three other different proofs also followed afterwards (see \cite{KiselevNazarov}, \cite{ConstantinVicol} and \cite{ConstantinTarfuleaVicol}).

The absence of the dissipative term $(-\Delta)^\gamma \omega$ in (\ref{viscousSQG}) (which tends to regularize the solution) together with the same Biot-Savart law (\ref{SQGBS}) leads to the inviscid SQG equation
\begin{align*} \label{equation}
\partial_t \omega+u\cdot\nabla \omega=0 \\
u=\nabla^{\perp} (-\Delta)^{-1/2} \omega 
\end{align*}
The inviscid SQG is a close relative of the 2D Euler equation which can be seen clearly if we consider the following family of singular integral type Biot-Savart law bridges in between (\ref{EulerBS}) and (\ref{SQGBS})
\begin{equation} \label{BSinterpolation}
u=\nabla^{\perp}(-\Delta)^{-1+\beta}\omega,\quad 0< \beta < 1/2
\end{equation}
This family of models are called modified SQG equations. Global regularity of the inviscid SQG equation or any member of the modified SQG equations also remain challenging and open. 

It is also noteworthy to mention the famous (inviscid) Boussinesq equations (\ref{Bouss1})(\ref{Bouss2})(\ref{Bouss3}) which models large scale atmospheric and oceanic flows that cause cold fronts and jet flows (see \cite{Majda} for details). 
\begin{align}\label{Bouss1}
\partial_t \omega+u\cdot \nabla \omega&=\partial_{x_1}\theta,\quad \omega(x,0)=u_0(x)\\ \label{Bouss2}
\partial_t \theta+u\cdot \nabla \theta &=0,\quad \theta(x,0)=\theta_0(x) \\ \label{Bouss3}
u&=\nabla^{\perp}(-\Delta)^{-1}\theta
\end{align}
The significance of the Boussinesq equation is, on the one hand, that after switching to cylindrical coordinates and introducing a change of variable, it can be identified, away from the rotation axis, with the 3D axisymmetric Euler equation with swirl. On the other hand, such 2D hydrodynamics model retains the key feature (i.e. vorticity stretching) of the 3D Euler equation and the Navier-Stokes equation. Due to these reasons, the global regularity of solutions of the Boussinesq equation is also outstandingly difficult. 

However, a recent numerical investigation of Luo and Hou \cite{HouLuo1} has shed important light on a potential singularity formation scenario for 3D Euler equation. Moreover, their work inspires a series of work which greatly help people understand a wide range of hydrodynamics models. Luo and Hou's original setup was the axisymmetric 3D Euler equation in a infinite vertical cylindrical domain with a periodic boundary condition in $z$ and no flux on boundary. The initial condition was chosen to have non-zero odd (in $z$) swirl $u^\theta$ and zero angular vorticity $\omega^\theta$. The numerical evolution of these initial data results in a quick growth of $\omega^\theta$ near the circle of hyperbolic points of the flow lying at the intersection of the boundary and $z=0$.

Kiselev-\v{S}ver\'{a}k \cite{KiselevSverak} used this hyperbolic growth scenario to construct a 2D Euler flow on a disk where the gradient of the vorticity grows at a double exponential rate . Due to their work, the double exponential upper bounds (which go back to Wolibner and H\"oder) of the growth of the gradient of the vorticity is now known to be sharp. In 2015, Kiselev-Ryzhik-Yao-Zlato\v{s} \cite{KiselevLenyaYao} utilized this idea to prove singularity formation for patch evolution for a modified SQG equation.

In order to study the full Boussinesq equation, Kiselev-Tan \cite{KiselevTan} considered a hyperbolic variation where they replaced $\partial_{x_1}\theta$ with $\frac{\theta}{x_1}$ in (\ref{Bouss1}) and worked with a modified version of (\ref{Bouss3}) and get
\begin{align} \label{modifiedBouss1}
  \partial_t \omega+u\cdot \nabla \omega&=\frac{\theta}{x_1},\quad \omega(x,0)=\omega_0(x)\\ \label{modifiedBouss2}
\partial_t \theta+u\cdot \nabla \theta &=0,\quad \theta(x,0)=\theta_0(x) \\ 
u(x,t)=(-x_1 \int_{y_1y_2\geq x_1x_2}\frac{\omega(y,t)}{|y|^2}dy,&x_2\int_{y_1y_2\geq x_1x_2}\frac{\omega(y,t)}{|y|^2}dy) \label{hyperbolicBoussBS}
\end{align}
The Biot-Savart law (\ref{hyperbolicBoussBS}) is modified from the asymptotic formula for $u$ described in \cite{KiselevSverak}. They proved blow up of solutions for this modified model. This turns out to be the first blow up result among all non-local active scalar incompressible flow model. In the same paper, they also proved global regularity of (\ref{hyperbolicEuler})(\ref{hyperbolicEulerBS}) (essentially (\ref{modifiedBouss1})(\ref{modifiedBouss2})(\ref{hyperbolicBoussBS}) with $\theta(x,t)=\theta_0(x)\equiv 0$)
\begin{align} \label{hyperbolicEuler}
  \partial_t u+u\cdot \nabla u&=0,\quad u(x,0)=u_0(x)\\ 
u(x,t)=(-x_1 \int_{y_1y_2\geq x_1x_2}\frac{\omega(y,t)}{|y|^2}dy,&x_2\int_{y_1y_2\geq x_1x_2}\frac{\omega(y,t)}{|y|^2}dy) \label{hyperbolicEulerBS}
\end{align}
Independently and simultaneously, Hoang-Orcan-Radosz-Yang \cite{hyperbolicBoussinesq} proved blow up of the same equation (\ref{modifiedBouss1})(\ref{modifiedBouss2}) with a different Biot-Savart law
\begin{equation} \label{compressibleBoussBS}
u(x,t)=(-x_1 \int_{S_{\alpha}}\frac{\omega(y,t)}{|y|^2}dy,x_2\int_{S_{\alpha}}\frac{\omega(y,t)}{|y|^2}dy)
\end{equation}
where the $S_\alpha=\{(x_1,x_2):0<x_1,0<x_2<\alpha x_1\}$ is a sector in the first quadrant with arbitrary large $\alpha$ as a parameter. The main difference of the two similar models lies in the incompressibility of (\ref{hyperbolicEulerBS}) and compressibility of (\ref{compressibleBoussBS}). But the form of (\ref{compressibleBoussBS}) is closer to the original one presented in the pioneer work \cite{KiselevSverak}.

Our goal in this paper is to analyze the following model for the SQG equation
\begin{align}
\partial_t\omega+u\cdot\nabla\omega=0,\quad \omega(x,0)=\omega_0(x) \label{hyperbolicSQG}\\
u(x,t)=(-x_1 \int_{y_1y_2\geq x_1x_2}\frac{\omega(y,t)}{|y|^{2+\alpha}}dy,&x_2\int_{y_1y_2\geq x_1x_2}\frac{\omega(y,t)}{|y|^{2+\alpha}}dy) \qquad (0<\alpha< 1)\label{hyperbolicSQGBS}
\end{align}

In spirit of (\ref{EulerBS}),(\ref{SQGBS}) and (\ref{BSinterpolation}), the two models (\ref{hyperbolicEuler})(\ref{hyperbolicEulerBS}) and (\ref{hyperbolicSQG})(\ref{hyperbolicSQGBS}) are closely related and differ only in the SIO-type Biot-Savart law by an extra parameter. For the rest of the paper, we use the notation 
\begin{align*}
\Omega(x,t)&=\int_{y_1y_2\geq x_1x_2}\frac{\omega(y,t)}{|y|^{2+\alpha}}dy\\ 
R&=\{(x_1,x_2)|x_1\geq 0, x_2\geq 0\}\\
C_{0}^{1}(R)&=\{f\in C^1(R): \supp(f) \text{ is compact}\}
\end{align*} 
We will first prove local well-posedness for this model and then show blow up for a wide range of initial data. In particular, the following theorems will be proved.
  
\begin{theorem} \label{locallywellposed}
Suppose $\omega_0\in C^1_0(R)$. There exists $T=T(\omega_0,\alpha)$ such that the hyperbolic SQG equation (\ref{hyperbolicSQG})(\ref{hyperbolicSQGBS})  admits a unique solution $\omega(x,t)\in C(0,T;C_0^1(R))$.
\end{theorem}
\begin{theorem} \label{blowup}
There exists smooth initial data $\omega_0$ such that the corresponding solution $\omega(x,t)$ of (\ref{hyperbolicSQG})(\ref{hyperbolicSQGBS}) blows up in finite time. Specifically, the finite time blow up holds in the sense that $\int_{0}^{t}\| \nabla \omega(\cdot,s)\|_{L^{\infty}}ds$ becomes infinite when $t$ reaches the maximum existence time $T$.
\end{theorem}

\section{local wellposedness of solutions}
Local existence of solutions is proved via Picard's Theorem on Banach Spaces. The flow map $X(x)=(X_1(x),X_2(x))$ of (\ref{hyperbolicSQG}) can be derived from \begin{equation} \label{flowmap1}
\frac{dX}{dt}(x,t)=u(X(x,t),t),\quad X(x,0)=x 
\end{equation}
with the velocity field $u$ given by Biot-Savart law (\ref{hyperbolicSQGBS}). Along the particle trajectory, $\omega_0$ is transported 
\begin{equation}\label{transported}
\omega(x,t)=\omega_0(X^{-1}(x,t))
\end{equation}
It suffices to show that ODE $(\ref{flowmap1})$ can be solved local in time in an open subset of appropriately chosen Banach space.
For convenience, let us set the following notation
\begin{align} 
\label{boundconstants1} n&:=\min\{x_1\;| (x_1,x_2)\in \supp \om_0\}\\
\label{boundconstants2} N&:=\max\{x_1\;| (x_1,x_2)\in \supp \om_0\}\\
M&:=\max\{x_2\;| (x_1,x_2)\in \supp \om_0\}
\end{align}
and define $\mathcal{B}=C^1(\mathbb{R}^2)$ and its subsets $O_{\delta}=\{X\in\mathcal{B}: X=\Id+\hat{X},  \|\hat{X}(\cdot,t)\|_{C^1(R)}< n-\delta, \inf_{x} \det(\nabla X)(x)>1/3\}$ where $\delta<n$ is a parameter.

\begin{corollary} \label{openSubset}
For any $\delta<n$, $O_{\delta}$ is a non-empty open subset of $\mathcal{B}$ that consists of local-homeomorphisms.
\end{corollary}
\begin{proof}
$O_{\delta}$ is non-empty since it contains identity and its multiples $cI$ for appropriate $c$. Due to continuity of the maps $\inf_x,  \det, \nabla$, the preimage $(\inf\circ\det\circ\nabla)^{-1}(1/3,\infty)$ is open. Due to continuity of $\|\cdot\|_{C^1}$ and the shift by identity $S_{Id}$, $S_{Id}[ (\|\cdot\|_{C^1})^{-1}(0,n-\delta)]$ is also open. Therefore, $O_{\delta}$ is open. By inverse function theorem, as all members of $O_{\delta}$ are $C^1$, they must all be local-homeomorphisms. 
\end{proof}

We quote the following lemma by Hadamard and recall a calculus inequality, both of which can be found in \cite{MajdaBertoxxi}.

\begin{lemma} \label{homeomorphism1}
Suppose that $X\in \mathcal{B}$ is a local homeomorphism and there exists $c$ such that $\|(\nabla X)^{-1}\|_{L^\infty }\leq c$, then $X$ is a homeomorphism of $\mathbb{R}^2$ onto $\mathbb{R}^2$.
\end{lemma}

\begin{lemma}\label{homeomorphism2}
Let $X:\mathbb{R}^d\to\mathbb{R}^d$ be a smooth, invertible transformation with $ |\det(\nabla X)(x)|>c$ for some $c>0$ and all $x$, then $$\|(\nabla X)^{-1}\|_{C^1}\leq C\|\nabla X\|_{C^1}^{2d-1}$$
\end{lemma}

By combining \textbf{Corollary \ref{openSubset}} and \textbf{Lemma \ref{homeomorphism1}, \ref{homeomorphism2}}, we see that $O_{\delta}$ is an open subset consists of bijective global homeomorphisms on $\mathbb{R}^2$.

\begin{lemma} \label{boundslemma}
 For all $X\in O_{\delta}$, there exists $n',N',M'>0$ such that 
\begin{align} \label{suppbound}
\supp (\om)\subset [n',N']\times [0,M'] 
\end{align}
\end{lemma}
\begin{proof}
Take $X\in O_\delta$. In the light of $(\ref{transported})$, $\supp(\omega)\subset X([n,N]\times M,t)$.  Besides, by triangle inequality 
$$n-(n-\delta) \leq X_1(x,t)=x_1+\hat{X_1}(x,t)\leq N+(n-\delta)$$ Argue similarly to get bounds for $X_2$. Combine the two bounds and the result follows from taking $0<n'\leq \delta, N'\geq N+n+\delta$ and $M'\geq M+n+\delta$.   
\end{proof}

\begin{lemma} \label{lipschitzlemma}
The velocity $u$ defined in $(\ref{hyperbolicSQGBS})$ satisfies
$$|u(X)-u(Y)|\leq C|X-Y|$$ 
for all $X,Y\in O_{\delta}$ with $C$ independent of the choice of $X, Y$. 
\end{lemma}
\begin{proof} By \textbf{Lemma \ref{boundslemma}}, we need only to show that $\nabla u$  
\begin{equation}\label{gradU}
\nabla u=\begin{pmatrix} -\Omega-x_1\Omega_{x_1}& x_2\Omega_{x_1}\\ -x_1\Omega_{x_2}& \Omega+x_2\Omega_{x_2}
\end{pmatrix}
\end{equation}
is uniformly bounded. Uniform boundedness of $\|\Omega\|_{L^{\infty}}$ is obvious as $supp(\omega)$ is uniformly bounded away from the origin. And
\begin{align*}
|\partial_{x_1}\Omega|=\int_{0}^{\infty} \frac{\omega(y_1,\frac{x_1x_2}{y_1},t)}{[\big(\frac{x_1x_2}{y_1}\big)^2+y_1^2}]^{1+\alpha/2}dy_1 
&\leq \int_{n'}^{N'} \frac{\omega(y_1,\frac{x_1x_2}{y_1},t)}{[\big(\frac{x_1x_2}{y_1}\big)^2+y_1^2}]^{1+\alpha/2}dy_1\\
&\leq  \|\omega\|_{L^\infty}\int_{n'}^{N'}  \frac{1}{[\big(\frac{x_1x_2}{y_1}\big)^2+y_1^2]^{1+\alpha/2}}dy_1\\
&= \|\omega_0\|_{L^\infty} \int_{n'}^{N'}  \frac{1}{[\big(\frac{x_1x_2}{y_1}\big)^2+y_1^2]^{1+\alpha/2}}dy_1\\
&\leq C(\omega_0,n',N',M')
\end{align*}
Entries involving $\partial_{x_2}\Omega$ can be bounded in a similar fashion. In the end, direct application of Picard's theorem completes the proof.
\end{proof}

\section{Finite Time Blow Up Of solutions}
Before showing blow up, let us first prove an analogue of the well-known Beal-Kato-Majda criterion which will serve as diagnostics for continuation of the solution and the blow up. For the rest of the paper, the constant $C$ may change from line to line but the dependence is only on $\omega_0$ and $\alpha$. We will also take advantage of the constants defined in (\ref{boundconstants1}) and (\ref{boundconstants2}).
\begin{proposition} \label{BKM}
Suppose $\omega\in C(0,T;C^1_0(R))$ solves (\ref{modelinnewcoordinate1})(\ref{modelinnewcoordinate2}) with initial data $\omega_0\in C_0^1(R)$ which does not identically vanish on $x_1$ axis and is bounded away from the origin. If 
$$\int_0^T\|\nabla\omega(\cdot,t)\|_{L^{\infty}}dt<\infty$$
then the solution $\omega$ can be continued to $[0,T+t_0)$ for some $t_0>0$. Otherwise, if $T$ is the largest time of existence of such solution $\omega$, then we must necessarily have
$$\lim_{t\to T}\int_{0}^{t}\|\nabla\omega(\cdot,t)\|_{L^{\infty}}dt=\infty$$
\end{proposition}
\begin{proof}
By Continuation Theorem of ODE on Banach Spaces, it suffices to show that $\|X(\cdot,t)\|_{C^1}$ is a priori controlled by $\int_{0}^{t}\|\nabla\omega(\cdot,t)\|_{L^{\infty}}dt$. Due to the compact support of $\omega$ (which can be seen from the compact support of $\omega_0$ and incompressibility of $u$), it suffices to show such control for $\|\nabla X(\cdot,t)\|_{L^{\infty}}$. Differentiating and taking $L^\infty$ norm in the particle trajectory map in (\ref{flowmap1}) gives
$$\frac{d}{dt}\|\nabla X(\cdot,t)\|_{L^{\infty}} \leq \|\nabla u(\cdot,t)\|_{L^{\infty}}\|\nabla X(\cdot,t)\|_{L^{\infty}}$$
Applying Gr\"{o}nwall's  inequality we get
$$\|\nabla X(\cdot,t)\|_{L^{\infty}}\leq e^{\int_{0}^{t}\|\nabla u(\cdot,s)\|_{L^{\infty}}ds}$$

By expression (\ref{gradU}), we then show each entry of $\nabla u$ will be a priori controlled by $\|\nabla\omega(\cdot,t)\|_{L^{\infty}}$. First note that we can safely dispel the case where $a=x_1x_2$ is large for two reasons. Firstly the compactness of $supp(\omega)$; secondly, structure of the integrand of $\Omega$ which suggests that the tail values are unharmful and the control near singularity is more vital. So we can assume $0\leq a\ll 1$. Now, since the flow map $u$ points to negative $x_1$ direction and positive $x_2$ direction for all $t$, one must have $N_t=\max\{x_1|(x_1,x_2)\in supp(\omega)\}<N$ for all $t$. Thus using mean value theorem and the elementary mean inequalities, we have 
\begin{align*}
|\Omega| \leq C\int_{y_1y_2\geq a}\frac{\|\nabla\omega(\cdot,t)\|_{L^{\infty}}|y|}{|y|^{2+\alpha}}dy
&= C\|\nabla\omega(\cdot,t)\|_{L^{\infty}}\int_0^N dy_1\int_{\frac{a}{y_1}}^{\infty}\frac{1}{(y_1^2+y_2^2)^{(1+\alpha)/2}}dy_2\\
&\leq C\|\nabla\omega(\cdot,t)\|_{L^{\infty}}\int_{\frac{a}{\xi}}^{\infty}\frac{1}{(\sqrt{\xi^2+y_2^2})^{1+\alpha}}dy_2\\
&\leq C\|\nabla\omega(\cdot,t)\|_{L^{\infty}}\int_{\frac{a}{\xi}}^{\infty}\frac{1}{(\xi+y_2)^{1+\alpha}}dy_2\\
&= C\|\nabla\omega(\cdot,t)\|_{L^{\infty}}(\xi+\frac{a}{\xi})^{-\alpha}\\
&\leq C\|\nabla\omega(\cdot,t)\|_{L^{\infty}} \bigg(\frac{1}{\frac{1}{\xi}+\frac{\xi}{a}}\bigg)^{-\alpha}\\
&= C\|\nabla\omega(\cdot,t)\|_{L^{\infty}} \bigg(\frac{a\xi}{a+\xi^2}\bigg)^{\alpha}\\
&\leq C\|\nabla\omega(\cdot,t)\|_{L^{\infty}} \bigg(\frac{a^2+\xi^2}{a+\xi^2}\bigg)^{\alpha}  \leq C\|\nabla\omega(\cdot,t)\|_{L^{\infty}} 
\end{align*}
where $0< \xi< N_t<N$. Next, the scaling argument $y_1^2=az$ gives 
\begin{align*}
|x_1\Omega_{x_1}| =x_1\partial_{x_1} \int_{0}^{\infty} dy_1\int_{\frac{a}{y_1}}^{\infty}\frac{\omega(y_1,y_2,t)}{(a^2+y_1^2)^{1+\alpha/2}}dy_2
&\leq x_1\int_{0}^{N}\frac{x_2}{y_1}\frac{\omega(y_1,\eta,t)}{((\frac{a}{y_1})^2+y_1^2)^{1+\alpha/2}}dy_1\\
&\leq C\|\nabla\omega(\cdot,t)\|_{L^{\infty}}\int_{0}^{N}\frac{a}{((\frac{a}{y_1})^2+y_1^2)^{1+\alpha/2}}dy_1\\
&\leq C\|\nabla\omega(\cdot,t)\|_{L^{\infty}}a^{(1-\alpha)/2}\int_{0}^{\infty}\frac{z^{(1+\alpha)/2}}{(1+z^4)^{1+\alpha/2}}dz\\
&\leq C\|\nabla\omega(\cdot,t)\|_{L^{\infty}}
\end{align*}
Other terms in $\nabla u$ can be estimated in the same manner. Last, observe that the solution can be continued as long as $\supp(\omega)$ stays away from $x_1=0$. But this will remain true if and only if $\nabla \omega$, which controls $\nabla u$, stays bounded. In turn, a bounded $\nabla u$ will guarantee that $\supp(\omega)$ cannot arrive at $x_1=0$ in finite time. 
\end{proof}
To show blow up, we start with changing coordinates $z_1=\log(x_1x_2), z_2=\log(\frac{x_1}{x_2})$ and writing $\tilde{\Omega}(z_1,t)=\Omega(x(z),t), \tilde{\omega}(z,t)=\omega(x(z),t)$, $\tilde{\omega}_0(z)=\omega_0(x(z))$. In $z$-coordinate, the model (\ref{hyperbolicSQG})(\ref{hyperbolicSQGBS}) can be rewritten as
\begin{equation} \label{modelinnewcoordinate1}
\partial_t\tilde{\omega}+2\tilde{\Omega}\partial_{z_2} \tilde{\omega}=0
\end{equation}
where 
\begin{equation} \label{modelinnewcoordinate2}
\tilde{\Omega}(z_1,t)=\frac{1}{4}\int_{z_1}^{\infty}e^{-\alpha y_1/2}dy_1\int_{-\infty}^{\infty}\frac{\tilde{\omega}(y,t)}{(\cosh y_2)^{1+\alpha/2}}dy_2
\end{equation}
Particularly, all particle trajectories now point to the positive direction of $z_2$-axis due to non-negativity of $\omega_0$ and the trajectory equation in integral form 

\begin{equation} \label{trajectoryintegral}
\tilde{X}(z_1,t)=\frac{1}{2}\int_{0}^{t}ds\int_{z_1}^{\infty}e^{-\alpha y_1/2}dy_1\int_{-\infty}^{\infty}\frac{\tilde{\omega}_0(y_1,y_2-\tilde{X}(y_1,s))}{(\cosh y_2)^{1+\alpha/2}}dy_2
\end{equation}

Set $F(z_1,t)= z_1+\tilde{X}(z_1,t)$ and $Z(t)=\max\{z_1|F(z_1,t)=0\}$.
The following lemmas will guarantee that $Z(t)$ is well-defined. 
\begin{lemma} \label{initialboundedbelow}
Suppose that $\omega_0\in\mathcal{C}_0^1(R)$ and $\omega_0$ is non-negative, not identically zero on $x_1$-axis. Then there exists $Z_1$ such that for every $z_1\leq Z_1$, we have 
$$\int_{-\infty}^{\infty}\tilde{\omega}_0(z_1,z_2)dz_2\geq C >0$$
\end{lemma}

\begin{proof}
Switching back to $x-$coordinates, we have
\begin{align*}
\int_{-\infty}^{\infty}\tilde{\omega}_0(z_1,z_2)dz_2=2\int_{0}^{\infty}\frac{1}{x_1}\omega_0(x_1,\frac{e^{z_1}}{x_1})dx_1
\end{align*}
Observe that the integral is taken on a hyperbolic section in the support of $\omega$. Since $\omega_0$ is $C^1$, hence as $z_1$ decreases to $-\infty$, the integral will converge to a positive number due to positivity of $\omega_0$. This means that for small enough $z_1$ the integral will be uniformly bounded away from $zero$. So the choice of appropriate $Z_1$ is possible.
\end{proof}

Fix $Z_1<\max \{z_1| z\in\supp \omega_0\}$ to be the maximum for which \textbf{Lemma \ref{initialboundedbelow}} holds. 
\begin{lemma} \label{welldefined}
Assume $\omega_0\in\mathcal{C}_0^{1}(R)$. Then for any $z_1<Z_1$
$$\lim_{t\to\infty}\tilde{X}(z_1,t)=\infty$$
\end{lemma}

\begin{proof}
Suppose by contradiction that there exists $z_1'<Z_1$ such that $\tilde{X}(z_1',t)\leq B<\infty$ for all $t$. Note that because of $\tilde{X}(z,t)=\int_0^{t}\tilde{\Omega}(z_1,s)ds$ and the positivity of $\omega_0$, $\tilde{X}(z_1,t)$ is monotonically decreasing in $z_1$. Hence $\tilde{X}(z_1,t)\leq B$ for all $z_1'<z_1<Z_1$. Using the fact $\supp \tilde{\omega}_0\subset\{(z_1,z_2)| z_1-2\log N_1\leq z_2\leq z_1-2\log n_1\}$ and \textbf{Lemma \ref{initialboundedbelow}}, it is easy to see that for every $y_1\in(z_1',Z_1)$
\begin{align} \label{lowerbound}
\int_{-\infty}^{\infty}\frac{\tilde{\omega}(y_1,y_2,t)}{(\cosh y_2)^{1+\alpha/2}}dy_2&=\int_{-\infty}^{\infty}\frac{\tilde{\omega}_0(y_1,y_2-\tilde{X}(y_1,t))}{(\cosh y_2)^{1+\alpha/2}}dy_2 \geq C e^{-(1+\alpha/2)(y_1+B)} >0
\end{align}
Thus 
$$
\tilde{X}(z_1' ,t)=\int_{0}^{t}\tilde{\Omega}(z_1',s)ds\geq C\int_{0}^{t}\int_{z_1'}^{Z_1} e^{-(1+\alpha)y_1-(1+\alpha/2)B}dy_1ds>Ct\to \infty
$$
All constants $C>0$ in above estimate depend only on $\omega_0,z_1'$ and $\alpha$. So we arrive at a contradiction.  
\end{proof}
We are now well-prepared to prove \textbf{Theorem \ref{blowup}}.
\begin{proof}
Based on the definition of $Z(t)$ and \textbf{Lemma \ref{welldefined}}, it is easy to see that $Z(t)$ is monotonically decreasing thus $a.e.$ differentiable. Then we know
\begin{equation} \label{differentialequality}
\frac{d}{dt}F(Z(t),t)=\partial_{z_1}F(Z(t),t)Z'(t)+\partial_{t}F(Z(t),t)=0\quad a.e.\quad t
\end{equation}
which follows immediately from $F(Z(t),t)=0$. The fact that $\tilde{X}$ being monotonically decreasing in $z_1$ gives that for any $z_1$
\begin{equation} \label{lessthanone}
\partial_{z_1}F(z_1,t)=1+\partial_{z_1}\tilde{X}(z_1,t)\leq 1
 \end{equation}
Let us choose $t_0=\inf \{t>0:Z(t_0)+1\leq Z_1\}$.  Then for all $t> t_0$ and the corresponding $z_1\in[Z(t),Z(t+1)]$, we have $0\leq F(z_1,t)\leq 1$. Now apply \textbf{Lemma \ref{initialboundedbelow}} to estimate as follows
\begin{align} \label{greaterthanzero}
\frac{d}{dt}F(Z(t),t)=\frac{d}{dt} \tilde{X}(Z(t),t) &=\frac{1}{2}\int_{Z(t)}^{\infty}e^{-\alpha y_1/2}dy_1\int_{-\infty}^{\infty}\frac{\tilde{\omega}_0(y_1,y_2-\tilde{X}(y_1,t))}{(\cosh y_2)^{1+\alpha/2}}dy_2\\ \nonumber 
&\geq \frac{1}{2}\int_{Z(t)}^{Z(t)+1}e^{-\alpha y_1/2}dy_1\int_{-\infty}^{\infty}\frac{\tilde{\omega}_0(y_1,y_2-\tilde{X}(y_1,t))}{(\cosh y_2)^{1+\alpha/2}}dy_2 \\ \nonumber
&\geq C \int_{Z(t)}^{Z(t)+1}e^{-\alpha y_1/2} \frac{1}{(\cosh F(z_1,t))^{1+\alpha/2}}dy_1\\ \nonumber
&\geq C \int_{Z(t)}^{Z(t)+1}e^{-\alpha y_1/2}dy_1 \geq C(\alpha,\omega_0)e^{-\alpha Z(t)/2}>0
\end{align}
Combining (\ref{differentialequality}),(\ref{lessthanone}) and (\ref{greaterthanzero}) yields 
$$Z'(t)\leq -\partial_{t}F(Z(t),t) \leq -Ce^{-\alpha Z(t)/2}\quad a.e. \quad t> t_0$$
The solution of above differential inequality shows that $Z(t)$ reaches $-\infty$ in finite time $T$. According to how $Z(t)$ is defined, this would in turn lead to the conclusion that $|\tilde{X}(z_1,t)|$ becomes infinite in finite time. Then the Beal-Kato-Majda criterion proved in \textbf{Proposition \ref{BKM}} demonstrates that the blow up happens in the sense of $\lim_{t\to T}\int_{0}^{t}\|\nabla \omega(\cdot,t)\|_{L^{\infty}}dt=\infty$.
\end{proof}
As an ending remark, let us relate back to the picture (see below). The above proof tells us that in $z-$coordinates all particles (shaded in black in the left figure) in $\supp(\tilde{\omega}_0)$ and to the left of $Z_1$, no matter how far down they are, will travel infinite distances, with vertical trajectories to arrive at $z_1$ axis in finite time. In the $x-$coordinate, $z_1$ axis becomes to the line $x_1=x_2$ and the corresponding particles (shaded in black in the right figure) will now travel with hyperbolic trajectories to cross $x_1=x_2$ in finite time. In particular, the particles in $\supp \omega_0$ lying on $x_1$ axis will travel horizontally and arrive at the origin in finite time. The continuation criterion breaks as a consequence of these particles bringing the support of $\omega$ onto $x_1=0$. 

\begin{figure}[H]
\centering
\begin{minipage}{.4\textwidth}
  \centering
  \includegraphics[width=.7\linewidth]{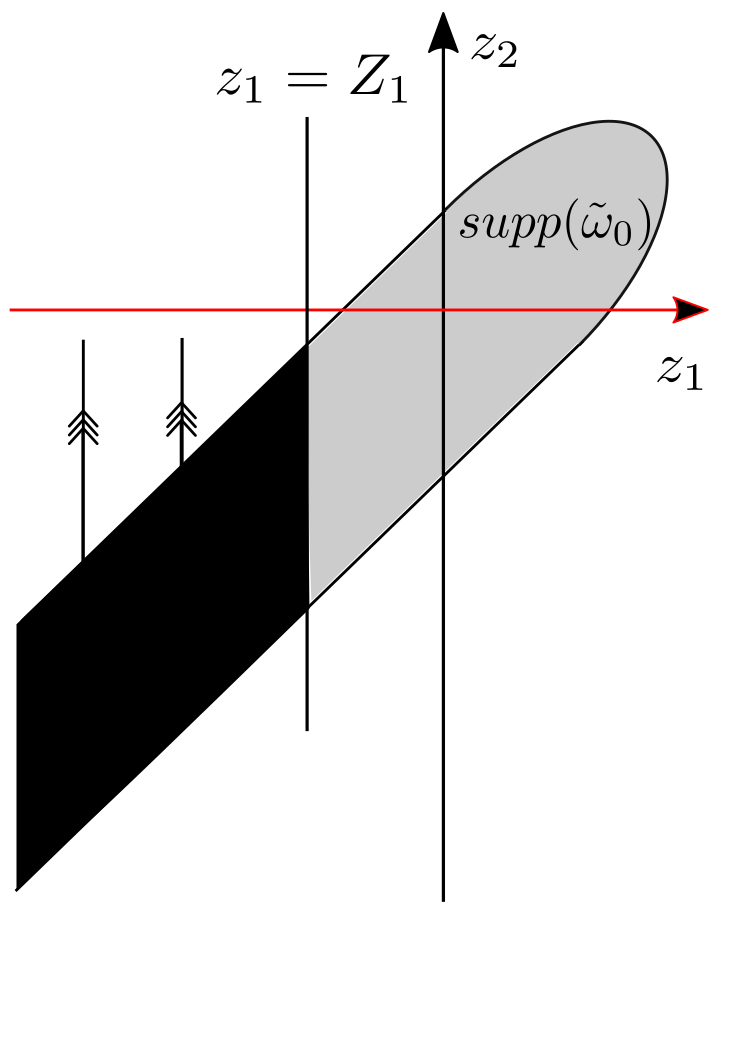}
  \captionof{figure}{In z-coordinates}
\end{minipage}%
\begin{minipage}{.4\textwidth}
  \centering
  \includegraphics[width=.7\linewidth]{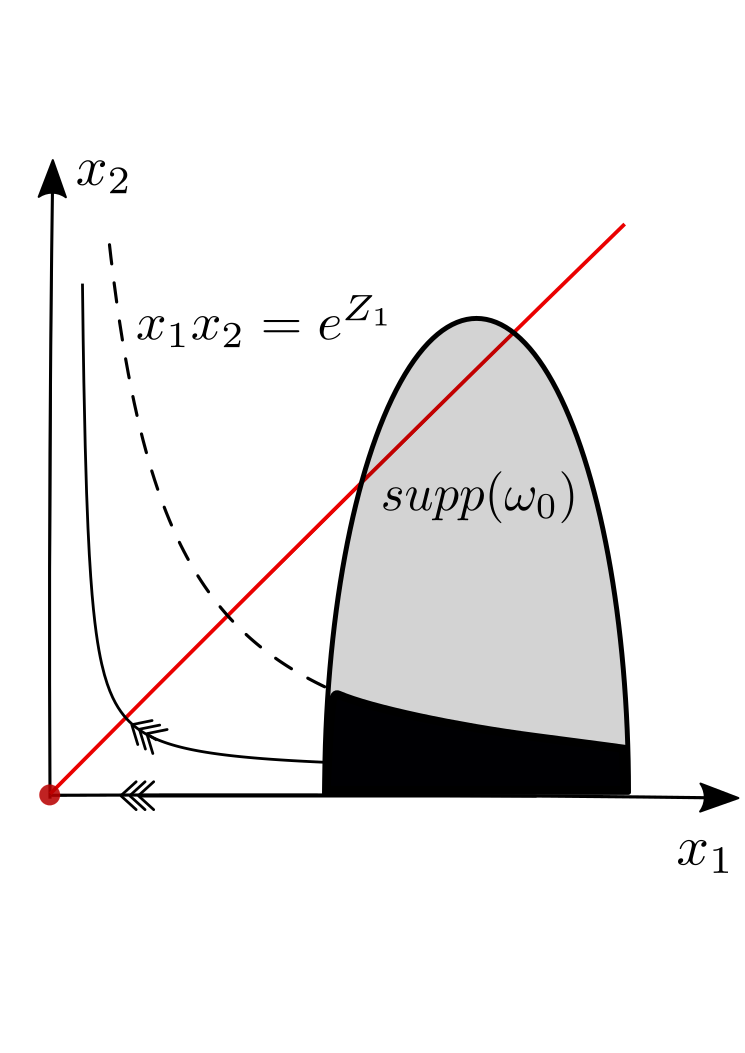}
  \captionof{figure}{In x-coordinates}
\end{minipage}
\end{figure}

\section{Acknowledgement}
I would like to thank Professor Alexander Kiselev for helpful discussions and valuable suggestions throughout the project and thank Department of Mathematics at Duke University for its hospitality. I would also acknowledge support of NSF-DMS award 1712294.

\end{document}